\documentclass[11pt,a4paper]{article}

\usepackage{inputenc}
\usepackage{amsfonts,amsthm,amsmath,amscd,amssymb}

\def\h{ {\cal H} }
\def\l{ {\cal L} }
\def\v{ {\cal V} }

\def\u{ {\cal U} }

\def\b{ {\cal B} }
\def\u{ {\cal U} }

\def\ii{ {\cal I} }

\def\s{ {\cal S} }
\def\e{ {\cal E} }
\def\p{ {\cal P} }
\def\r{ {\cal R} }

\def\hh{\mathbb{H}}

\newtheorem{teo}{Theorem}[section]
\newtheorem{prop}[teo]{Proposition}
\newtheorem{lem}[teo]{Lemma}
\newtheorem{coro}[teo]{Corollary}
\newtheorem{defi}[teo]{Definition}
\theoremstyle{definition}
\newtheorem{rem}[teo]{Remark}

\title{Canonical sphere bundles of the Grassmann manifold.}

\begin{document}

\date{}
\author{E. Andruchow, E. Chiumiento and G. Larotonda}

\maketitle 

\begin{abstract}
For a given Hilbert space $\mathcal H$, consider the space of self-adjoint projections $\mathcal P(\mathcal H)$. In this paper we study the differentiable structure of a canonical sphere bundle over $\mathcal P(\mathcal H)$ given by
$$
\mathcal R=\{\, (P,f)\in \mathcal P(\mathcal H)\times \mathcal H \, :  \, Pf=f , \, \|f\|=1\, \}.
$$
We establish the smooth action on $\mathcal R$ of the group of unitary operators of $\mathcal H$, therefore $\mathcal R$ is an homogeneous space. Then we study the metric structure of $\mathcal R$ by endowing it first with the uniform quotient metric, which is a Finsler metric, and we establish minimality results for the geodesics. These are given by certain one-parameter groups of unitary operators, pushed into $\mathcal R$ by the natural action of the unitary group. Then we study the restricted bundle $\mathcal R_2^+$ given by considering only the projections in the restricted Grassmannian, locally modelled by Hilbert-Schmidt operators. Therefore we endow $\mathcal R_2^+$ with a natural Riemannian metric that can be obtained by declaring that the action of the group is a Riemannian submersion. We study the Levi-Civita connection of this metric and establish a Hopf-Rinow theorem for $\mathcal R_2^+$, again obtaining a characterization of the geodesics as the image of certain one-parameter groups with special speeds. \footnote{{\bf 2010 MSC: 22E65, 47B10, 58B20}.

{\bf Keywords: sphere bundle, Finsler metric, Riemannian metric, geodesic, projection, flag manifold}.}
\end{abstract}

%22E65 Infinite-dimensional Lie groups and their Lie algebras
%58B20 Riemannian, Finsler and other geometric structures
%47B10 Operators belonging to operator ideals (nuclear, $p$-summing, in the Schatten-von Neumann classes, etc.)

\bigskip

\section{Introduction}

Let $\h$ be a complex Hilbert space, $\b(\h)$ the algebra of bounded operators and $\p(\h)$ the set of all orthogonal projections on $\h$.  Throughout, we identify $\p(\h)$ with the Grassmann manifold of $\h$ (i.e. the set consisting  of all closed subspaces of $\h$). In this paper, we study the differential and metric structure of the following set
$$
\r=\{\, (P,f)\in\p(\h)\times \h \, :  \, Pf=f , \, \|f\|=1\, \},
$$
which is the total space of the canonical sphere bundle given by 
$
\pi_\r: \r\to\p(\h) , \ \ \pi_\r(P,f)=P.
$
Clearly, the unitary group  $\u(\h)$ acts on $\r$: 
$$
U\cdot (P,f)=(UPU^*,Uf), \, \, \, \, U\in\u(\h), \, \, (P,f)\in\r.
$$
As we shall see below,  the connected components of $\r$ turn out to be a smooth   homogeneous spaces of $\u(\h)$. Thus, the present work can be seen as a contribution to the geometry of infinite dimensional homogeneous spaces arising in operator theory and operator algebras, which has been a subject of study in different settings, for instance Grassmann manifolds \cite{al, BR07b, CPR93, K79, pr}, Stiefel manifolds \cite{ARV, Ch10} and orbits of selfadjoint operators \cite{AL10, BV17}.   Also abstract homogeneous spaces are considered in \cite{agc10, cocometricacociente, m-l r}; for more examples and a detailed account we refer to the book \cite{beltita}. On the other hand, the definition of the bundle $(\pi_\r, \r, \p(\h))$  is motivated from classical concepts in topology. Canonical bundles over finite dimensional Grassmann manifolds  played a fundamental role in classification problems  of vector bundles \cite{milnor}, meanwhile results on classification of sphere bundles goes back to the works \cite{S44, W40}.

The contents of this paper are as follows. In Section \ref{sec2}  we show that the connected components of $\r$ coincide with orbits given by the action defined above. This is indeed a direct consequence of the following fact: given $(P_0,f_0) \in \r$, the map 
$$
\rho_{(P_0,f_0)}: \u(\h)\to \r,  \, \, \, \rho_{(P_0,f_0)}(U)=(UP_0U^*,Uf_0)
$$ 
has continuous local cross sections. In Section \ref{sec3} we also use this result to prove that each orbit is a smooth homogeneous space of $\u(\h)$ and a submanifold of $\b(\h) \times \h$. The latter means that the quotient topology in $\r$ coincides with the topology inherited from $\b(\h) \times \h$, and that tangent spaces, which can be described as
$$
(T\r)_{(P_0,f_0)}=\{ \, (XP_0-P_0X,Xf_0)  \,  : \, X=-X^* \in \b(\h),    \, \}, 
$$
are closed and complemented in $\b(\h) \times \h$.  
Section \ref{sec4} contains the results on the metric geometry of $\r$. As a homogeneous space, the natural way to endow each component of $\r$ is to push down a metric from $\u(\h)$.  If we consider in $\u(\h)$ the Finsler metric defined by using the operator norm $\| \, \cdot \, \|$, then the Finsler metric on $\r$ is given as follows: for $(XP_0-P_0X,Xf_0) \in (T\r)_{(P_0,f_0)}$, 
$$
\| (XP_0-P_0X,Xf_0)  \|_{(P_0,f_0)} : = \inf \{ \,  \|  X + Y \| \,  : \, Y=-Y^*, \, YP_0=P_0Y,  \, Yf_0=0  \, \}.
$$ 
This is  the quotient norm in $(T\r)_{(P_0,f_0)} \simeq \b_{ah}(\h)/ \mathfrak{g}$, where $\b_{ah}(\h)$ is the space of anti-Hermitian operators and $\mathfrak{g}$ is the Lie algebra of the isotropy group.
Indeed, note that $ Y=-Y^*$, $YP_0=P_0Y$ and $Yf_0=0$ if and only if $Y \in \mathfrak{g}$.  
We show that the quotient norm is always attained, or in the terminology of \cite{cocometricacociente}, there exists a \textit{minimal lifting} for each tangent vector. To prove this, we observe that the quotient norm  can be interpreted as a two-step best approximation problem using $3 \times 3$ block operator matrices. In the first step, we use Krein's extension theorem (see \cite{krein, RN55, DKW82}).

The Finsler metric defined in $\r$ allows us to measure the length of curves. If $\gamma:[0,1] \to \r$ is a piecewise smooth curve, then its length is given by
$$
L(\gamma)=\int_0^1 \| \dot{\gamma}(t) \|_{\gamma(t)} \, dt.
$$
We prove that the initial value problem has a solution: given $(P_0,f_0) \in \r$ and a  tangent vector $(XP_0-P_0X,Xf_0) \in (T\r)_{(P_0,f_0)}$,  there exists $X_0^*=-X_0$ such that the curve 
$$
\delta(t)=e^{tX_0}\cdot (P_0,f_0)=(e^{tX_0}P_0e^{-tX_0}, e^{tX_0}f_0)
$$ 
satisfies $\delta(0)=(P_0,f_0)$, $\dot{\delta}(0)=(XP_0-P_0X,Xf_0)$ and $\delta(t)$ has minimal length in some interval. 
This problem was already solved in \cite{pr} for the Grassmann manifold $\p(\h)$, where the quotient Finsler metric coincides with the metric given by the usual operator norm. Remarkably, it was also solved for smooth homogeneous spaces which are the quotient of the unitary group of a von Neumann algebra by the unitary group of a sub-algebra \cite{cocometricacociente}. Although our homogeneous space does not fit in this category, our proof is adapted from that work.

Section \ref{sec5} is devoted to the Riemannian geometry of the canonical sphere bundle of the restricted (or Sato) Grassmannian \cite{al, BR07b, ps}. 
The total space of this bundle is
$$
\r_2^+=\{(P,f)\in\r: P\in\p_+^2(\h)\},
$$ 
where $\p_+^2(\h)$ is the connected component containing a fixed projection $P_+$ in the restricted Grassmannian. The group $\u_2(\h)$ of unitaries which are Hilbert-Schmidt perturbations of the identity acts transitively on $\p_+^2(\h)$. We observe  that $\r_2^+$ is also acted  upon  transitively by $\u_2(\h)$, and has structure of smooth homogeneous space of $\u_2(\h)$. It is a submanifold of $\b_2(\h) \times \h$, where $\b_2(\h)$ denotes the Hilbert-Schmidt operators on $\h$. We endow $\r_2^+$ with two different Riemannian metrics: the ambient metric and 
the reductive (or quotient) metric. These metrics are shown to be equivalent and complete. The geodesics of the reductive metric are computed. Any geodesic is minimal up to a critical time value (which is characterized). The main result in this section establishes an estimate of the geodesic radius, which is at least $\pi/4$ (points of $\r_2^+$ at distance less than $\pi/4$ are joined by a unique minimal geodesic).

\section{Transitivity and local cross sections for the action}\label{sec2}

Given $P_0\in\p(\h)$, we define
$$
\r_{P_0}=\{(P,f)\in\r: P \hbox{ is unitary equivalent to } P_0\}.
$$
It is well known that $P$ is unitary equivalent to $P_0$ if and only if $\dim R(P_0)=\dim R(P)$ and $\dim N(P_0)=\dim N(P)$. 
\begin{prop}\label{secloc}
Let $P_0\in\p(\h)$. The action of  $\u(\h)$ on $\r_{P_0}$ is transitive. 
If $f_0\in R(P_0)$, then the map
$$
\rho_{(P_0,f_0)}: \u(\h)\to \r_{P_0} , \ \  \rho_{(P_0,f_0)}(U)=(UP_0U^*,Uf_0)
$$ 
has continuous local cross sections.
\end{prop}
\begin{proof}
Let $f_0\in R(P_0)$ with $\|f_0\|=1$, and $(P,f)\in\r_{P_0}$. There exists $U\in\u(\h)$ such that $UP_0U^*=P$. Then $U^*f$ and $f_0$ belong to the unit sphere of  $R(P_0)$. Then there exists a unitary operator  $V_0$ in $\u(R(P_0))$ such that $V_0f_0=U^*f$. $V_0$ can be extended to a unitary operator $V$ in $\h$, defining it as the identity on $R(P_0)^\perp$. Note that $V$ commutes with $P_0$.
Then $W=UV$ satisfies
$$
Wf_0=UVf_0=UV_0f_0=UU^*f=f \hbox{ and } WP_0W^*=UVP_0V^*U^*=UP_0U^*=P.
$$ 
The same procedure provides continuous local cross sections for $\rho_{(P_0,f_0)}$. Consider the 
open subset of $\p(\h)$:
$$
\{P\in\p(\h): \|P-P_0\|<1\}.
$$
It is well known that if $\|P-P_0\|<1$ then there exists $\mu_{P_0}(P)$ a unitary operator which is continuous (and smooth) in the parameter $P$, such that 
$$
\mu_{P_0}(P_0)=1 \hbox{ and } \mu_{P_0}(P)P_0\mu_{P_0}^*(P)=P.
$$
Then one has that the following set 
$$
\b_{(P_0,f_0)}=\{(P,f)\in\r: \|P-P_0\|<1 \hbox{ and } \|\mu_{P_0}(P)^*f-f_0\|<2\}
$$
is a open subset of $\r$ (and of $\r_{P_0}$). Moreover, for any pair of unit vectors $\xi,\eta$ in an arbitrary Hilbert space $\h$, satisfying that $\|\xi-\eta\|<2$, there exists $\nu_{\xi,\eta}\in\u(\h)$, which is a continuous and smooth map in terms of $\xi$ and $\eta$,  such that $\nu_{\xi,\eta}\xi=\eta$. Explicit formulas for $\nu$ can be obtained in several ways. For instance, if  $x,y,z\in\h$, and  $(x\otimes y)(z)=\langle z,y\rangle x$  denotes the elementary rank one operator, we  use  $\nu=\nu_{\xi,\eta}$ the unitary operator given by rotation in the plane generated by $\xi,\eta$
$$
\nu  =1+(c-1)\xi\otimes\xi +(\overline{c}-1)\xi'\otimes\xi'+s(\xi'\otimes\xi-\xi\otimes\xi')=\left(\begin{array}{ccc}
c & s& 0 \\
-s & \overline{c} &0\\
0& 0& 1
\end{array}
\right),
$$
where $c=\langle\eta,\xi\rangle$, $s=\sqrt{1-|c|^2}>0$ (assuming $\|\xi-\eta\|<2$), and $\xi'=\frac{1}{s}(\eta-c\xi)$.

\bigskip

If $(P,f)\in\b_{(P_0,f_0)}$, then $f_0$ and $\mu_{P_0}(P)^*f$ are unit vectors in $R(P_0)$. Let $V_{(P_0,f_0)}(P,f)$ be the unitary operator which acts as $\nu_{f_0,\mu_{P_0}(P)^*f}$ on $R(P_0)$ and as the identity in $R(P_0)^\perp$. Then 
$$
\sigma_{(P_0,f_0)}:\b_{(P_0,f_0)}\to \u(\h), \ \ \sigma_{(P_0,f_0)}(P,f)=\mu_{P_0}(P)V_{(P_0,f_0)}(P,f)
$$
is a continuous local cross section for $\rho_{(P_0,f_0)}.$ 
\end{proof}
\begin{coro}
The connected components of $\r$ are $\r_{P_0}$, $P_0\in\p(\h)$, where two projections $P_0$ and $P_1$ define the same component if  and only if they are unitarily equivalent. 
\end{coro}

\section{Regular structure}\label{sec3}

Let us prove that  $\r$ is a $C^\infty$ differentiable manifold, and that the map $\pi_\r$ is a  $C^\infty$ fibre bundle. In order to establish this assertion, the following lemma will be useful (see \cite{rae}).
\begin{lem}\label{raeburn}
Let $G$ be a Banach-Lie group acting smoothly on a Banach space $X$. For a fixed
$x_0\in X$, denote by $\rho_{x_0}:G\to X$ the smooth map $\rho_{x_0}(g)=g\cdot
x_0$. Suppose that
\begin{enumerate}
\item
$\rho_{x_0}$ is an open mapping,  regarded as a map from $G$ onto the orbit
$\{g\cdot x_0: g\in G\}$ of $x_0$ (with the relative topology of $X$).
\item
The differential $d(\rho_{x_0})_1:(TG)_1\to X$ splits: its nullspace and range are
closed complemented subspaces.
\end{enumerate}
Then the orbit $\{g\cdot x_0: g\in G\}$ is a smooth (analytic) submanifold of  $X$, and the
map
$$
\rho_{x_0}:G\to \{g\cdot x_0: g\in G\}
$$ 
is a smooth submersion.
\end{lem}
Denote by $\b_{ah}(\h)=\{X\in\b(\h): X^*=-X\}$, the space of anti-Hermitian operators, which is the Banach-Lie algebra of $\u(\h)$, and let $\b_{h}(\h)=\{X\in\b(\h): X^*=X\}$ be the space of selfadjoint (or Hermitian) operators. 

\begin{prop}\label{smooth}
$\r$ is a $C^\infty$ submanifold of $\b(\h)\times\h$. For any $(P_0,f_0)\in\r$, the map 
$$
\rho_{(P_0,f_0)}:\u(\h)\to\r_{P_0} 
$$
is a  $C^\infty$-submersion, and in fact all manifolds and maps are real analytic.
\end{prop}
\begin{proof}
By Proposition \ref{secloc}, $\rho=\rho_{(P_0,f_0)}$  has a local continuous cross section $\sigma$ such that \\ $\sigma((P_0,f_0))=1$, and $\rho\circ\sigma=id_V$ where $V\subset \r_{P_0}$ is an open neighbouhood of $(P_0,f_0)$. Let $Z\subset \u(\h)$ be an open neighbourhood of $1$, note that $\rho^{-1}(\rho(Z))=ZK$ where $K$ is the isotropy group of $\rho$. Therefore
$$
\rho(Z)=\sigma^{-1}\rho^{-1}\rho(Z)=\sigma^{-1}(ZK)
$$   
is open in $\r_{P_0}$ since the section $\sigma$ is continuous and $ZK=\cup_{k\in K} Zk$ is open in $\u(\h)$. This proves that $\rho$ is locally open around $1\in\u(\h)$; by the  transitivity of the action of $\rho$ in $\r_{P_0}$, it follows that $\rho:\u(\h)\to\r_{P_0}$ is an open mapping.

Now by elementary computations,  $\delta_{(P_0,f_0)}=d(\rho_{(P_0,f_0)})_1:\b_{ah}(\h)\to \b(\h)\times \h$
is given by
$$
\delta_{(P_0,f_0)}(X)=(XP_0-P_0X,Xf_0).
$$
The nullspace of $\delta_{(P_0,f_0)}$ consists of anti-Hermitian operators $X$ which commute with $P_0$ and satisfy $Xf_0=0$. Since $f_0\in R(P_0)$, this set consists of $2\times 2$ matrices in terms of the decomposition $\h=R(P_0)\oplus N(P_0)$ which are of the form
$$
X=\left( \begin{array}{cc} X_{11} & 0 \\ 0 & X_{22} \end{array} \right),
$$
with $X_{ii}$ anti-Hermitian and $X_{11}f_0=0$. Furthermore, $X_{11}$ can be written as a matrix in terms of the decomposition $R(P_0)=<f_0>\oplus <f_0>^\perp$,
$$
X_{11}=\left( \begin{array}{cccc} 0 & 0 &  0 & \dots \\ 0 & x'_{11} & x'_{12} & \dots \\ 0 & x'_{21} & x'_{22} & \dots \\  \vdots & \vdots & \vdots & \ddots \end{array} \right) =\left( \begin{array}{cc} 0 & 0 \\ 0 & X' \end{array} \right),
$$
where $X'$ is an  anti-Hermitian operator acting in $R(P_0)\ominus <f_0>$.
A natural supplement for the nullspace $\delta_{(P_0,f_0)}$, as matrices in the decomposition 
$$
\h=<f_0>\oplus (R(P_0)\ominus<f_0>) \oplus N(P_0)
$$
is the set of matrices of the form
\begin{equation}\label{nucleoparaerre}
\left( \begin{array}{cc}
\begin{array}{cc}  i  t & \vec{f} \\  -\vec{\bar{f}}^t  & 0 \end{array}   &  Y \\  -Y^*    & 0 \end{array} \right),
\end{equation}
with $t\in\mathbb{R}$.

The range of $\delta_{(P_0,f_0)}$ is given by
$$
(T\r)_{(P_0,f_0)}=\{ \,  (XP_0-P_0X,Xf_0)\,   : \, X=-X^* \, \} \subseteq \b_h(\h)\times\h. 
$$
To prove that this subspace is complemented in $\b(\h)\times\h$, it is sufficient to prove that
it is complemented in $\b_h(\h)\times\h$. This can be done using the projection $\mathcal{E}: \b_h(\h)\times\h \to \b_h(\h) \times \h$ defined by
\begin{equation}\label{E}
\mathcal{E}((X,h))=(P_0^\perp XP_0+P_0XP_0^\perp, i\text{Im}<h,f_0>f_0 + (P_0 - (f_0 \otimes f_0) )h + P_0^\perp Xf_0).
\end{equation}
Clearly, $\mathcal{E}$ is continuous, linear, and it is not difficult to check that $\mathcal{E}^2=\mathcal{E}$ and 
$$
\mathcal{E}((XP_0-P_0X,Xf_0))=(XP_0-P_0X,Xf_0)
$$
for any anti-Hermitian operator $X$. Let us prove that the range of $\mathcal{E}$ is contained in $(T\r)_{(P_0,f_0)}$. Given $X=X^*$, $h \in \h$, we have to construct an operator $Y=-Y^*$ such that $\mathcal{E}((X,h))=((YP_0-P_0Y,Yf_0))$. First we use the following fact: given vectors $f, g$ in a Hilbert space $\l$, with $\|f\|=1$ there is an operator $z \in \b_{ah}(\l)$ such that $zf=g$ if and only if 
$<g,f>=-<f,g>$. In this case, the operator $z$ can be taken as  
$$
z=g\otimes f  + <f,g>f \otimes f-f\otimes g.
$$
Using this fact with $f=f_0$ and $g=i\text{Im}<h,f_0>f_0 + (P_0 - (f_0 \otimes f_0) )h$, we have an operator $z \in \b_{ah}(R(P_0))$ satisfying $zf_0= i\text{Im}<h,f_0>f_0 + (P_0 - (f_0 \otimes f_0) )h$.
Now suppose  that the matrix  of $X$  with respect to $P_0$ is
$$
X=\begin{pmatrix}  x_{11}  & x_{12} \\ x_{12}^*  &  x_{22} \end{pmatrix},
$$
then take 
$$
Y=\begin{pmatrix}  z  & -x_{12} \\ x_{12}^*  & 0 \end{pmatrix}.
$$
It follows that $YP_0-P_0Y=P_0^\perp XP_0+P_0XP_0^\perp$, and
\begin{align*}
Yf_0 & = P_0Y P_0 f_0 +  P_0^\perp Y P_0 f_0  =   z f_0 + P_0^\perp Xf_0\\
& = i\text{Im}<h,f_0>f_0 + (P_0 - (f_0 \otimes f_0) )h + P_0^\perp Xf_0. 
\end{align*}
This proves that  $\mathcal{E}$ is a projection with range $(T\r)_{(P_0,f_0)}$.
\end{proof}

\begin{rem}
The range of $\delta_{(P_0,f_0)}$ consists of pairs $(XP_0-P_0X,Xf_0)$, the left hand part of this pair is a selfadjoint operator which is co-diagonal with respect to $P_0$, $P_0ZP_0=P_0^\perp ZP_0^\perp=0$. It is easy to see that any selfadjoint operator  with this property is of the form $[X,P_0]$ for some $X^*=-X$. Indeed, the conditions are equivalent to $Z=ZP_0+P_0Z$ as seen from the computation
$$
Z=ZP_0+(P_0^{\perp}+P_0)ZP_0^{\perp}=ZP_0+P_0^{\perp}ZP_0^{\perp}+P_0ZP_0^{\perp}=ZP_0+0+P_0Z-P_0ZP_0=ZP_0+P_0Z.
$$
Therefore if we take $X=ZP_0-P_0Z=[Z,P_0]$, then clearly $X^*=-X$ and $[X,P_0]=ZP_0+P_0Z=Z$ as claimed. Note that $X$ is not unique, it can be modified by adding to it any skew-adjoint operator commuting with $P_0$.
\end{rem}

\medskip

\begin{prop}
The map 
$$
\pi_\r:\r\to\p(\h), \, \, \, \pi_\r(P,f)=P 
$$
is a  locally trivial fibre bundle. 
\end{prop}
\begin{proof}
Fix $(P_0,f_0)\in\r$, and consider the open subsets
$$
\v_{(P_0,f_0)}=\{(P,f):\in\r: \|P-P_0\|<1\} \hbox{ and } \v_{P_0}=\{P\in\p(\h): \|P-P_0\|<1\} 
$$ 
and the map
$$
\v_{(P_0,f_0)}\to \v_{P_0}\times \mathbb{S}(R(P_0)), (P,f)\mapsto (P,\mu_{P_0}^*(P)f),
$$
where $\mu_{P_0}$ is the $\u(\h)$-valued map defined in the proof of Proposition \ref{secloc}. Clearly $\mu_{P_0}^*(P)f\in R(P_0)$, and therefore this map is well defined, with inverse $(P,h)\mapsto (P,\mu_{P_0}(P)h)$, which is a local trivialization of the map $\pi_\r$.
\end{proof}

\section{The quotient metric in $\r$}\label{sec4}

In \cite{cocometricacociente} Dur\'an, Mata-Lorenzo and Recht presented a program to study homogeneous spaces which are quotients of the unitary group of a C$^*$-algebra by the unitary group of a sub-C$^*$-algebra. This program does not apply exactly to the homogeneous space $\r$, since the isotropy subalgebra of $\u(\h)$ is not the unitary group of a C$^*$-algebra. Nevertheless, the main ideas of their approach apply in our context with minor modifications. 

Dur\'an, Mata-Lorenzo and Recht start with a natural idea: since $\r$ is a quotient, the tangent space $(T\r)_{(P_0,f_0)}$ identifies with the quotient
$$
(T\r)_{(P_0,f_0)}\simeq \b_{ah}(\h) / \ker \delta_{(P_0,f_0)}=\b_{ah}(\h) / \{ X\in\b_{ah}(\h): [X,P_0]=0, Xf_0=0\}.
$$
This is a quotient of Banach spaces, so one endows it with the quotient norm: if $V\in (T\r)_{(P_0,f_0)}$
\begin{equation}\label{norma de R}
|V|_{(P_0,f_0)}=\inf\{ \|X\|: X\in\b_{ah}(\h),  \delta_{(P_0,f_0)}(X)=V\}.
\end{equation}
The  elements $X$ in $\b_{ah}(\h)$ such that $\delta_{(P_0,f_0)}(X)=V$, will be called {\it liftings} of $V\in (T\r)_{(P_0,f_0)}$. Let us define  minimal liftings:
\begin{defi}
Let $V\in (T\r)_{(P_0,f_0)}$. A lifting $X_0$ of $V$ is called a {\it minimal lifting} if it achieves the quotient norm, i.e.
$$
\|X_0\|=\inf\{\|X\|: X \hbox{ is a lifting of } V\}=|V|_{(P_0,f_0)}.
$$
\end{defi}
It can be proved that minimal liftings exist for any tangent vector. We shall postpone the proof of this fact, before let us further characterize minimal liftings.

Given a vector $V\in (T\r)_{(P_0,f_0)}$, by the form of the kernel of $\delta_{(P_0,f_0)}$ described in Section 3, it is clear that different liftings of $V$ have the following common form:
\begin{equation}\label{matriz}
X_*=\left( \begin{array}{cc}
\begin{array}{cc}  i x_0 & \vec{x} \\  -\vec{\bar{x}}^t  & * \end{array}   &  A \\  -A^*    & ** \end{array} \right),
\end{equation}
where $x_0\in\mathbb{R}$,  and  $*$ and $**$ are to be filled with anti-Hermitian operators in $R(P_0)\ominus <f_0>$ and $N(P_0)$, respectively. Therefore the problem of finding a minimal lifting consists in filling the spots $*$ and $**$ in order that the completed matrix has the least possible norm. This problem has certain resemblance to M.G. Krein's extension problem (see \cite{krein}, \cite[p. 336]{RN55} or \cite{DKW82}). In fact, as it will be shown next, it can be tackled using this method. In our context, Krein's extension problem consists in completing the $*$ spot in the matrix operator (in terms of the decomposition $\h=R(P_0)\oplus N(P_0)$)
$$
\left( \begin{array}{cc} B & A \\ -A^* & * \end{array} \right)
$$
in order that the matrix has the least possible norm. We shall divide our quest into two steps. 
\begin{enumerate}
\item
First fill the space $*$ in the first row of $X_*$:
\begin{equation}\label{fila1}
\left( \begin{array}{cc} \begin{array}{cc} i x_0 & \vec{x} \\ \vec{\bar{x}}^t & *  \end{array} \ \   \ \  A \end{array}  \right)
\end{equation}
with $*=Y_0$, in order that the row has the least possible norm.
\item
Next, fill the $**$ spot in the matrix
\begin{equation}\label{matriz2}
\left( \begin{array}{cc}
\begin{array}{cc}  i x_0 & \vec{x} \\  -\vec{\bar{x}}^t  & Y_0 \end{array}   &  A \\  -A^*    & ** \end{array} \right)
\end{equation}
\end{enumerate}
\begin{prop}
Suppose that  $Y_0$ and $Z_0$ are operators acting respectively in $R(P_0)\ominus <f_0>$ and $N(P_0)$, where $Y_0$ optimizes the norm of the row (\ref{fila1}), and $Z_0$ optimizes the norm of the matrix (\ref{matriz2}).  Then 
\begin{equation}\label{matriz posta}
X_0=\left( \begin{array}{cc}
\begin{array}{cc}  i x_0 & \vec{x} \\  -\vec{\bar{x}}^t  & Y_0 \end{array}   &  A \\  -A^*    & Z_0 \end{array} \right)
\end{equation}
is a minimal lifting. Any minimal lifting can be obtained in this fashion.
\end{prop}
\begin{proof}
By hypothesis, for any anti-Hermitian operator $Y$ in $R(P_0)\ominus <f_0>$, one has that
$$
\left\|\left( \begin{array}{cc} \begin{array}{cc} i x_0 & \vec{x} \\ \vec{\bar{x}}^t & Y_0  \end{array} \ \   \ \  A \end{array}  \right)\right\| \le \left\|\left( \begin{array}{cc} \begin{array}{cc} i x_0 & \vec{x} \\ \vec{\bar{x}}^t & Y  \end{array} \ \   \ \  A \end{array}  \right)\right\|.
$$
Since $Z_0$ is a solution of Krein's extension problem, for our optimal choice of  $Z_0$ we have
$$
\left\|\left( \begin{array}{cc}
\begin{array}{cc}  i x_0 & \vec{x} \\  -\vec{\bar{x}}^t  & Y_0 \end{array}   &  A \\  -A^*    & Z_0 \end{array} \right)\right\|=
\left\|\left( \begin{array}{cc} \begin{array}{cc} i x_0 & \vec{x} \\ \vec{\bar{x}}^t & Y_0  \end{array} \ \   \ \  A \end{array}  \right)\right\|.
$$
On the other hand, for any block matrix operator, the norm of the first row is less or equal than the norm of the full matrix:
$$
\left\|\left( \begin{array}{cc} \begin{array}{cc} i x_0 & \vec{x} \\ \vec{\bar{x}}^t & Y  \end{array} \ \   \ \  A \end{array}  \right)\right\| \le 
\left\| \left( \begin{array}{cc}
\begin{array}{cc}  i x_0 & \vec{x} \\  -\vec{\bar{x}}^t  & Y \end{array}   &  A \\  -A^*    & Z \end{array} \right) \right\|,
$$
for any anti-Hermitian operator $Z$ in $N(P_0)$, which proves our first assertion.

Let $X_1$ be a minimal lifting, 
$$
\left( \begin{array}{cc}
\begin{array}{cc}  i x_0 & \vec{x} \\  -\vec{\bar{x}}^t  & Y_1 \end{array}   &  A \\  -A^*    & Z_1 \end{array} \right).
$$
Then, if $Y_0$ and $Z_0$ are obtained as above
$$
\left\|\left( \begin{array}{cc} \begin{array}{cc} i x_0 & \vec{x} \\ \vec{\bar{x}}^t & Y_1  \end{array} \ \   \ \  A \end{array}  \right)\right\|\le \left\|\left( \begin{array}{cc}
\begin{array}{cc}  i x_0 & \vec{x} \\  -\vec{\bar{x}}^t  & Y_1 \end{array}   &  A \\  -A^*    & Z_1 \end{array} \right)\right\|=
\left\|\left( \begin{array}{cc}
\begin{array}{cc}  i x_0 & \vec{x} \\  -\vec{\bar{x}}^t  & Y_0 \end{array}   &  A \\  -A^*    & Z_0 \end{array} \right)\right\|.
$$
This last norm equals 
$$
\left\|\left( \begin{array}{cc} \begin{array}{cc} i x_0 & \vec{x} \\ \vec{\bar{x}}^t & Y_0  \end{array} \ \   \ \  A \end{array}  \right)\right\|.
$$
From these inequalities it follows  that also $Y_1$ minimizes the norm of the first row, and therefore also $Z_1$ is a solution of Krein's extension problem for the matrix
$$
\left( \begin{array}{cc}
\begin{array}{cc}  i x_0 & \vec{x} \\  -\vec{\bar{x}}^t  & Y_1 \end{array}   &  A \\  -A^*    & ** \end{array} \right).
$$
\end{proof}
Therefore to obtain minimal liftings, it suffices to solve the first row problem (\ref{fila1}). Let us show that solutions to this problem always exist.
\begin{prop}
There exists an anti-Hermitian operator $Y_0$ in $R(P_0)\ominus <f_0>$ such that 
$$
\left\|\left( \begin{array}{cc} \begin{array}{cc} i x_0 & \vec{x} \\ \vec{\bar{x}}^t & Y_0  \end{array} \ \   \ \  A \end{array}  \right)\right\| \le \left\| \left( \begin{array}{cc} \begin{array}{cc} i x_0 & \vec{x} \\ \vec{\bar{x}}^t & Y  \end{array} \ \   \ \  A \end{array}  \right) \right\|
$$
for any other anti-Hermitian operator $Y$ in  $R(P_0)\ominus <f_0>$.
\end{prop}
\begin{proof}
Denote by $\iota$ the infimum of all possible completions of the first row matrix (\ref{fila1}).
Let $Y_n$ be a minimizing sequence of anti-Hermitian operators in  $R(P_0)\ominus <f_0>$:
$$
\left\|\left( \begin{array}{cc} \begin{array}{cc} i x_0 & \vec{x} \\ \vec{\bar{x}}^t & Y_n  \end{array} \ \   \ \  A \end{array}  \right)\right\|\to \iota.
$$
Note that 
$$
\|Y_n\|\le \left\|\left( \begin{array}{cc} \begin{array}{cc} i x_0 & \vec{x} \\ \vec{\bar{x}}^t & Y_n  \end{array} \ \   \ \  A \end{array}  \right)\right\|.
$$
Therefore the sequence $Y_n$ is norm bounded, and therefore it has a weak operator convergent subsequence $Y_{n_k}$,  $<Y_{n_k}\xi,\eta>\to <Y_0\xi,\eta>$. Clearly $Y_0$ is anti-Hermitian. We claim that $Y_0$ is a solution to the problem. Indeed, for any $\epsilon>0$, there exists a unit vector $\xi_\epsilon$ such that $|<Y_0\xi_\epsilon,\xi_\epsilon>|\ge \|Y_0\|-\epsilon$.
Then 
$$
\|Y_{n_k}\|\ge |<Y_{n_k}\xi_\epsilon,\xi_\epsilon>|\to   |<Y_0\xi_\epsilon,\xi_\epsilon>|\ge \|Y_0\|-\epsilon.
$$
Then $\iota\ge \|Y_0\|-\epsilon$ for any $\epsilon$, and thus $\iota=\|Y_0\|$.
\end{proof}
The following result is again adapted from \cite{cocometricacociente}, with minor modifications which  allow to treat the case when the isotropy group of the unitary action is an arbitrary Banach-Lie group (and not the unitary group of a sub-C$^*$-algebra).

\begin{prop}\label{representacion}
Let $X_0$ be a lifting of $V\in (T\r)_{(P_0,f_0)}$. Then $X_0$ is a minimal lifting if and only if there exists a  representation $\pi$ of $\b(\h)$ on a Hilbert space $\l$, and a unit vector $\xi\in\l$ such that
\begin{enumerate}
\item
$\pi(X_0^2)\xi=-\|X_0\|^2\xi$
\item
$Re<\pi(X_0)\xi, \pi(B)\xi>=0$, for all $B\in\ker \delta_{(P_0,f_0)}$.
\end{enumerate}
\end{prop}
\begin{proof}
Suppose first that conditions 1. and 2. hold. Then for any $B\in\ker \delta_{(P_0,f_0)}$,
$$
\|X_0+B\|^2\ge \|\pi(X_0+B)\xi\|^2=\|\pi(X_0)\xi\|^2+\|\pi(B)\xi\|^2\ge \|\pi(X_0)\xi\|^2=-<\pi(X_0^2)\xi,\xi>=\|X_0\|^2.
$$
Conversely, suppose that $X_0$ is a minimal lifting. Let $\s$ be the real subspace of $\b_h(\h)$ spanned by
$X_0^2+\|X_0\|^21$ and the operators of the form $BX_0+X_0B$, for $B\in \ker \delta_{(P_0,f_0)}$ (note that since $B$ and $X_0$ are anti-Hermitian,  these operators are indeed selfadjoint). We claim that the minimality condition of $X_0$, implies that 
$$
\s\cap Gl(\h)^+=\emptyset.
$$
Suppose that this intersection is non empty, i.e. there exists $s\in\mathbb{R}$ and $b\in\ker \delta_{(P_0,f_0)}$ such that
$$
s(X_0^2+\|X_0\|^21)+BX_0+X_0B \ge r 1
$$
for some $r>0$. We may suppose $s>0$, otherwise add $(1-s)(X_0^2+\|X_0\|^2)$ to the above inequality.
Dividing by $s$ we get
\begin{equation}\label{mayor que cero}
X_0^2+\|X_0\|^21+BX_0+X_0B \ge r 1
\end{equation}
for some (other) $B\in\ker\delta_{(P_0,f_0)}$ and $r>0$.
In particular this implies that
$$
X_0^2+BX_0+X_0B\ge (r-\|X_0\|^2)1,
$$
and therefore the spectrum of $X_0^2+BX_0+X_0B$ satisfies
$$
\sigma(X_0^2+BX_0+X_0B)\subset (-\|X_0\|^2,+\infty).
$$
Also, for $n\ge 1$,
$$
n(X_0^2+\|X_0\|^21)+BX_0+X_0B \ge X_0^2+\|X_0\|^21+BX_0+X_0B \ge r 1
$$
and thus
$$
X_0^2+\|X_0\|^21+\frac{B}{n}X_0+X_0\frac{B}{n} \ge\frac{r}{n} 1.
$$
Since $X_0^2+\frac{B}{n}X_0+X_0\frac{B}{n}\to X_0^2$, by the semi-continuity of the spectrum, given the open neighbourhood $(-\infty,\|X_0\|^2)$ of $X_0^2$ (recall that $X_0\leq 0$), there exists $n_0$ such that for $n\ge n_0$
$$
\sigma\left(X_0^2+\frac{B}{n}X_0+X_0\frac{B}{n}\right)\subset (-\infty,\|X_0\|^2).
$$
In particular, there exists $B\in\ker\delta_{(P_0,f_0)}$ of arbitrary small norm such that
$$
\sigma(X_0^2+BX_0+X_0B)\subset (-\|X_0\|^2,\|X_0\|^2),
$$
i.e. $\|X_0^2+BX_0+X_0B\|<\|X_0\|^2$.  We claim that this inequality contradicts the minimality of $X_0$. Indeed, this follows from the following result:
\begin{lem}
If $\|X_0+B\|\ge \|X_0\|$ for all $B\in\ker\delta_{(B_0,f_0)}$, then $\|X_0^2+BX_0+X_0B\|\ge \|X_0\|^2$ for all $B\in\ker\delta_{(B_0,f_0)}$.
\end{lem}
\begin{proof}
Consider for $t\in (0,1)$ the function
$$
g(t)=-X_0^2+\frac{1}{t}\{(X_0+tB)^*(X_0+tB)+X_0^2\}.
$$
We claim that $\|g(t)\|\ge \|X_0\|^2$. Otherwise, the convex combination
$$
tg(t)+(1-t)(-X_0^2)
$$  
would have norm strictly less than $\|X_0\|^2$. Note that this convex combination equals
$(X_0+tB)^*(X_0+tB)$. Thus we would have
$\|X_0\|^2>\|(X_0+tB)^*(X_0+tB)\|=\|X_0+tB\|^2$, which contradicts the minimality of $X_0$. Then
$$
\|-X_0^2-X_0B-BX_0-tB\|\ge \|X_0\|^2
$$
for all $t\in(0,1)$. If $t\to 0$ this yields
$$
\|X_0^2+X_0B+BX_0\|\ge \|X_0\|^2,
$$
which proves the lemma.
\end{proof}
Our claim follows: $\s\cap GL(\h)^+=\emptyset$. $Gl(\h)^+$ is a convex open subset of $\b_h(\h)$, $\s$ is a subspace of $\b_h(\h)$. Therefore, by the Hahn-Banach theorem there exists a (real) linear functional $\varphi_0$  in $\b_h(\h)$ such that
$$
\varphi_0(\s)=0 \hbox{ and } \varphi_0(Gl(\h)^+)>0.
$$
This functional extends to a complex linear functional $\varphi$ defined in $\b(\h)$. The fact that $\varphi(Gl^+(\h))>0$ implies that $\varphi$ is positive and bounded. We may normalize it so that it has norm $1$. Consider $\pi$  the Gelfand-Naimark-Segal representation induced by $\varphi$, with cyclic vector $\xi$. Let us prove that $\pi$ and $\xi$ satisfy the conditions in the statement of the proposition. Since $\varphi$ is zero at the generators of $\s$,
$$
\varphi(X_0^2+\|X_0\|^2)=0, \hbox{  i.e. } \varphi(X_0)^2=-\| X_0\|^2
$$
and
$$
\varphi(X_0B+BX_0)=0.
$$
The first condition means that $<\pi(X_0^2)\xi,\xi>=-\|X_0\|^2$.
Thus one has equality in the Cauchy-Schwarz inequality 
$$
\|X_0\|^2=<\pi(X_0^2)\xi,\xi>=|<\pi(X_0^2)\xi,\xi>|\le \|\pi(X_0^2)\|\|\xi\|^2\le \|X_0\|^2.
$$
Then $\pi(X_0^2)\xi$ is a multiple of $\xi$, and thus $\pi(X_0^2)\xi=-\|X_0\|\xi$. The second condition
$$
0=\varphi(X_0B+BX_0)=\varphi(X_0B)+\varphi((X_0B)^*)=\varphi(X_0B)+\overline{\varphi(X_0B)},
$$
i.e.
$$
0= Re \varphi(X_0B)=Re <\pi(X_0B)\xi,\xi>=-Re <\pi(B)\xi,\pi(X_0)\xi>.
$$ 
Considering $iB\in\ker\delta_{(P_0,f_0)}$, one also gets that  $Im <\pi(B)\xi,\pi(X_0)\xi>=0$,
which finishes the proof.
\end{proof}

\begin{teo}
Let $V\in (T\r)_{(P_0,f_0)}$ with $|V|_{(P_0,f_0)}=1$. Let $X_0$ be a minimal lifting of $V$. Then the curve 
$$
\delta(t)=e^{tX_0}\cdot (P_0,f_0)=(e^{tX_0}P_0e^{-tX_0}, e^{tX_0}f_0)
$$
has minimal length along its path for $|t|\le \pi/2$.
\end{teo}
\begin{proof}
Following Dur\'an, Mata Lorenzo, and Recht \cite{cocometricacociente}, we shall construct a smooth map from $\r$, in fact, from the connected component of $(P_0,f_0)$ in $\r$ to the unit sphere of an appropriate Hilbert space.
Let $\pi_0$ and $\xi_0$ be the cyclic representation and the cyclic vector in the Hilbert space $\h_0$ obtained by means of Proposition \ref{representacion}. Denote by $\bar{\delta}_0$ the natural extension of $\delta_{(P_0,f_0)}$, which is defined in $\b_{ah}(\h)$, to the whole $\b(\h)$, namely: $\bar{\delta}_0(A)=(AP_0-P_0A,Af_0)$. Consider the closed subspace
$$
\s_0=\overline{\{ \pi_0(A)\xi_0: A\in\ker\bar{\delta}_0\}}\subset \h_0.
$$
and the symmetry $\rho_0$ acting in $\h_0$, which equals the identity in $\s_0$ and minus the identity in $\s_0^\perp$. 
Let $W$ be in the isotropy group $\ii_{(P_0,f_0)}$. Then $\pi_0(W)$ commutes with $\rho_0$, indeed, it leaves $\s_0$ invariant
$$
\pi_0(W)\pi_0(A)\xi_0=\pi_0(WA)\xi_0\in\s_0,
$$
because $WA\in\ker\bar{\delta}_0$.  Consider the map, defined in the connected component $\r_0$ of $(P_0,f_0)$ in $\r$,  to the set of symmetries of $\h_0$ (denoted appropriately in \cite{cocometricacociente} the Grassmann manifold $Gr(\h_0)$ of $\h_0$),
$$
F_0: \r_0\to Gr(\h_0) , \ \ F_0(U\cdot (P_0,f_0))=\pi_0(U) \rho_0 \pi_0(U)^*.
$$
$F_0$ is well defined: if $U_1\cdot (P_0,f_0)=U_2\cdot (P_0,f_0)$, then $U_2^*U_1\in\ii_{(P_0,f_0)}$. Then 
$$
\pi_0(U_2^*U_1) \rho_0=\rho_0 \pi_0(U_2^*U_1), \ \hbox{ i.e. } \pi_0(U_1)\rho_0 \pi_0(U_1)^*=\pi_0(U_2)\rho_0\pi_0(U_2)^*.
$$
Also $F_0$ is smooth: because of the manifold structure of $\r_0$ is final with respect to the smooth submersion $\rho=\rho_{(P_0,f_0)}$ (Proposition \ref{smooth}) it suffices to check that
$$
(F\circ\rho)(U)=\pi_0(U)\rho_0\pi_0(U)^*
$$ 
is smooth. This is apparent since $\pi_0$ is bounded and  linear, the product and the adjoint are smooth in $\b(\h_0)$, and $Gr(\h_0)$ is an embedded submanifold of $\b(\h_0)$.

There is a natural map from $Gr(\h_0)$ to the unit sphere $\mathbb{S}(\h_0)$, by means of $\xi_0$,
$$
\Pi_0:Gr(\h_0)\to \mathbb{S}(\h_0), \ \ \Pi_0(\rho)=\rho\xi_0.
$$
The proof of the theorem proceeds with the following syllogism:
\begin{enumerate}
\item
The  maps $\frac12 F_0$ and  $\Pi_0$ are length reducing.
\item
The composition $\Pi_0F_0$ exactly doubles the length of $\delta$ (i.e. length of $\Pi_0F_0\delta$ is $2$  times the length of $\delta$), and $\Pi_0F_0\delta$ is a minimal geodesic of $\mathbb{S}(\h_0)$.
\item
Therefore, $\delta$ has minimal length along its path, provided that $|t|\le \pi/2$.
\end{enumerate}
Let us check these steps. 

\smallskip

\noindent Step 1: Since $\Pi_0$ is the restriction of a linear map, if $\rho(t)$ is a curve in $Gr(\h_0)$ with $\rho(0)=\rho$ and $\dot{\rho}(0)=X$
$$
d(\Pi_0)_\rho(X)=\frac{d}{dt}\biggm|_{t=0}\Pi_0(\rho(t))=X\xi_0,  
$$
and thus $\|d(\Pi_0)_\rho(X)\|=\|X\xi_0\|\le \|X\|$. On the other hand, let $\gamma(t)=U(t)\cdot(P_0,f_0)$ be a smooth curve in $\r$, with $\gamma(0)=(P,f)=U\cdot(P_0,f_0)$, $\dot{\gamma}(0)=(X,v)=V$. Denote $u(t)=\pi_0(U(t))$ and $u=\pi_0(U)$. Then 
\begin{equation}\label{diferencial de F_0}
\frac{d}{dt}\biggm|_{t=0} F_0(\gamma(t))=\frac{d}{dt}\biggm|_{t=0}u(t)\rho_0u^*(t)=\dot{u}(0)\rho_0 u^* +u\rho_0 \dot{u}^*(0)=u\{u^*\dot{u}(0)\rho_0+\rho_0 \dot{u}^*(0)u\}u^*.
\end{equation}
Since $0=\dot{u^*u}=\dot{u}^*u+u^*\dot{u}$, and thus $\dot{u}^*u=-u^*\dot{u}$, (\ref{diferencial de F_0}) equals
$$
u\{u^*\dot{u}\rho_0-\rho_0u^*\dot{u}\}u^*, 
$$
whose norm equals
$$
\|u\{u^*\dot{u}\rho_0-\rho_0u^*\dot{u}\}u^*\|=\|u^*\dot{u}\rho_0-\rho_0u^*\dot{u}\|.
$$
Since $\rho_0=2P_{\s_0}-1$,
$$
u^*\dot{u}\rho_0-\rho_0u^*\dot{u}=[u^*\dot{u},\rho_0]=2[u^*\dot{u},P_{\s_0}],
$$
and therefore this commutant is a co-diagonal matrix with respect to the decomposition $\h_0=\s_0\oplus\s_0^\perp$. Then its norm equals the norm of the first column $u^*\dot{u}P_{\s_0}$.  Thus,
$$
\|d(F_0)_{(P,f)}(V)\|=2\|[u^*\dot{u},P_{\s_0}]\|=2\|u^*\dot{u}P_{\s_0}\|\le 2\|u^*\dot{u}(0)\|=2\|\pi_0(U^*\dot{U}(0))\|\le 2\|\dot{U}(0)\|.
$$
On the other hand, $(UP_0U^*,Uf_0)=\gamma(0)=(P,f)$ and 
$$
V=\dot{\gamma}(0)=(\dot{U}(0)P_0U^*+UP_0\dot{U}^*(0),\dot{U}(0)f_0)=(\dot{U}(0)U^*P+PU\dot{U}^*(0),\dot{U}(0)U^*f).
$$
Again using that $U\dot{U}^*(0)=-\dot{U}(0)U^*$, one has that
$$
V=(\dot{U}(0)U^*P-P\dot{U}(0)U^*,\dot{U}(0)U^*f)=\delta_{(P,f)}(\dot{U}(0)U^*),
$$
i.e. $\dot{U}(0)U^*\in\b_{ah}(\h)$ is a lifting for $V\in (T\r)_{(P,f)}$. 

We claim that any lifting $Z$ of $V$ is obtained in this fashion. Indeed, if $\delta_{(P,f)}(Z)=(ZP-PZ,Zf)=V$, consider $U(t)=e^{tZ}$ and $\gamma(t)=U(t)\cdot (P,f)\in\r$. Clearly $U(0)\cdot (P,f)=(P,f)$ and 
$\dot{U}(0)U^*=Z$.

It follows from this claim and inequality (\ref{diferencial de F_0}) that 
$$
\|d(F_0)_{(P,f)}(V)\|\le 2 \|Z\|
$$
for any lifting $Z$ of $V$. Then
$$
\|d(F_0)_{(P,f)}(V)\|\le 2 |V|_{(P,f)}.
$$

\smallskip

\noindent Step 2: Clearly, we have $\|d(\Pi_0F_0)_{\Pi_0F_0(P,f)}(X)\|\le 2 |V|_{(P,f)}$.
Let us compute the image of $\delta(t)=e^{tX_0}\cdot(P_0,f_0)$ under $\Pi_0F_0$:
$$
\Pi_0F_0(\delta(t))=\pi_0(e^{tX_0})\rho_0 \pi_0(e^{-tX_0})\xi_0=e^{t\pi_0(X_0)}\rho_0e^{-t\pi_0(X_0)}\xi_0.
$$
Note that since $\pi_0(X_0^2)\xi_0=-\|X_0\|^2\xi_0$, we have
$$
\pi_0(X_0)^{2n}\xi_0=\|X_0\|^{2n}\xi_0 \hbox{ and } \pi_0(X_0)^{2n+1}\xi_0=(-1)^n\|X_0\|^{2n}\pi_0(X_0)\xi_0
$$
for $n\ge 0$. Then  $e^{-t\pi_0(X_0)}\xi_0=\cos(t\|X_0\|)\xi_0-\frac{1}{\|X_0\|} \sin(t\|X_0\|)\pi_0(X_0)\xi_0$. Note that since $<\pi(X_0)\xi_0,\pi_0(A)\xi_0>=0$ for any $A\in\ker\bar{\delta}_0$, one has that $\pi_0(X_0)\xi_0\in\s_0^\perp$ (whereas $\xi_0\in\s_0$). Then $\rho_0\xi_0=\xi_0$ and $\rho_0\pi_o(X_0)\xi_0=-\pi_0(X_0)\xi_o$. Thus
$$
\rho_0e^{-t\pi_0(X_0)}\xi_0=\cos(t\|X_0\|)\xi_0+\frac{1}{\|X_0\|} \sin(t\|X_0\|)\pi_0(X_0)\xi_0=e^{t\pi_0(X_0)}\xi_0,
$$
and
$$
\Pi_0F_0\delta(t)=e^{t\pi_0(X_0)}\rho_0e^{-t\pi_0(X_0)}\xi_0=e^{2t\pi_0(X_0)}\xi_0=\cos(2t\|X_0\|)\xi_0+\frac{1}{\|X_0\|} \sin(2t\|X_0\|)\pi_0(X_0)\xi_0,
$$
which is a minimal geodesic of $\mathbb{S}(\h_0)$, up to $|2t|\le \pi$, i.e. $|t|\le \pi/2$. 
Finally, the length of $\Pi_0F_0\delta$ is $2\|\pi_0(X_0)\xi_0\|=2\|X_0\|$, by the properties of $\pi_0$ and $\xi_0$ established in the previous lemma. The proof finishes comparing the lengths of the images in the unit sphere (using the mapping $\Pi_0F_0$) of $\delta$ on any subinterval of $[-\pi/2,\pi/2]$ and any other curve in $\r$ joining the same endpoints. 
\end{proof}
\begin{rem}
Let us examine particular directions in $\r$, and compute minimal liftings (and thus minimal geodesics).
\begin{enumerate}
\item
Let $V=(0,g)\in (T\r)_{(P_0,f_0)}$. Then $P_0g=g$, and $<f_0,g>\in i\mathbb{R}$. Also note that 
$$
X_g=g\otimes f_0-f_0\otimes g+<f_0,g>f_0\otimes f_0
$$
satisfies $X_g^*=-X_g$ and $X_gf_0=g$. Since $g\in R(P_0)$, $X_g$ commutes with $P_0$:
$$
X_gP_0=g\otimes P_0f_0-f_0\otimes P_0g+<f_0,g>f_0\otimes P_0f_0=X_g=P_0X_g.
$$
Thus, $X_g$ is a lifting of $V$, i.e. $\delta_{(P_0,f_0)}(X_g)=([X_g,P_0],X_gf_0)=(0,g)=V$. To find minimal liftings, 
use the vector $f_0$ as the first vector of  an orthonormal basis $\{ f_i \}_{i=0}^\alpha$  of $R(P_0)$, where $\alpha=\dim R(P_0) \in [1, \infty]$. Note that if $e,e'\in\h$ are orthogonal to $f_0$, then $<X_ge,e'>=0$.  Thus, if $g=\sum_{i=0}^\alpha g_i f_i$, then the matrix of $X_g$ in the  decomposition $\h=\left(<f_0>\oplus (R(P_0)\ominus <f_0>)\right) \oplus N(P_0)$ is 
\[
X_g=
\left(
\begin{array}{c|c}
 \begin{array}{cr} \begin{array}{cccc} ig_0 & -\bar{g}_1 & -\bar{g}_2 & \dots \\ g_1 & 0 & 0 & \dots \\  g_2 & 0 & 0 & \dots \\ \vdots & \vdots & \vdots & \ddots \end{array} \\  \end{array} & 0 \\
\hline
0 & 0
\end{array}
\right)
\]
with $g_0\in\mathbb{R}$. Since the $1,2$, $2,1$ and $2,2$ block entries are zero, the second step stated in (\ref{matriz2}) has the trivial solution, so it remains to solve the best approximation problem in (\ref{fila1}). According to \cite[Corol. 1.3]{DKW82} all the solutions of this problem can be parametrized as
$$ Y_Z  = ig_0 T + \|g\|(1-T)^{1/2}Z (I-T)^{1/2},
$$
where $Z$ is any anti-Hermitian operator and
$$
T=\left(\|g\|^2-|g_0|^2\right)^{-1}\begin{pmatrix} g_1 \\ g_2 \\ \vdots   \end{pmatrix} \begin{pmatrix} \bar{g}_1 & \bar{g}_2 & \ldots \end{pmatrix}.
$$
We thus get that
\[
X_Z=
\left(
\begin{array}{c|c}
 \begin{array}{cr} \begin{array}{cccc} ig_0 & -\bar{g}_1 & -\bar{g}_2 & \dots \\ g_1 &  &  &  \\  g_2 &  & Y_Z &  \\ \vdots  &  &  &  \end{array} \\  \end{array} & 0 \\
\hline
0 & 0
\end{array}
\right)
\]
is a minimal lifting of $V$. Next note that the minimal curve $\delta_Z(t)=e^{tX_Z}f_0$ remains inside the unit sphere of $R(P_0)$: 
$$
\delta_Z(t)=(e^{tX_Z}P_0e^{-tX_Z}, e^{tX_Z}f_0)=(P_0,e^{tX_Z}f_0).
$$
If $Z=0$, then the minimal lifting $X_0$ may be expressed as
$$
X_0=\gamma_0 \otimes f_0 - f_0 \otimes \gamma_0  + i g_0 (f_0 \otimes f_0 + e \otimes e),
$$
where $e=\gamma_0 / \| \gamma_0 \|$. Since $X_1=\gamma_0 \otimes f_0 - f_0 \otimes \gamma_0$ and $X_2=i g_0 (f_0 \otimes f_0 + e \otimes e)$ commute, a straightforward computation shows that
\begin{align*}
\delta_0(t) & =(P_0,e^{tX_0}f_0)=(P_0, e^{tX_1}e^{tX_2}f_0) \\
& = (P_0, e^{itg_0}e^{tX_1}f_0)= (P_0, e^{itg_0}(\cos(t\|\gamma_0\|)f_0 + \sin(t\|\gamma_0\|)e)).
\end{align*}
Clearly, this curve is a geodesic in the unit sphere if and only if $g_0=0$. Therefore our metric, given by minimal liftings, does not induce the usual metric when restricted to the sphere.  
\item
Let $V=(Z,0)\in(T\r)_{(P_0,f_0)}$. Then $Zf_0=f_0$, with $Z^*=Z$  $P_0$-codiagonal. Consider $X_Z=ZP_0-P_0Z$. Then clearly $X_Z^*=-X_Z$ and $X_Zf_0=ZP_0f_0-P_0Zf_0=f_0-f_0=0$. Moreover,
$$
X_ZP_0-P_0X_Z=(ZP_0-P_0Z)P_0-P_0(ZP_0-P_0Z)=ZP_0-2P_0ZP_0+P_0Z=Z,
$$
because $Z$ is $P_0$-codiagobal. Thus $\delta_{(P_0,f_0)}(X_Z)=(Z,0)=V$. Also it is clear that $X_Z$ is a minimal lifting: it is codiagonal. Its norm coincides with the norm of $Z$. Finally, since $X_Zf_0=0$, the geodesic $\delta(t)=e^{tX_Z}f_0$ is
$$
\delta(t)=(e^{tX_Z}P_0e^{-tX_Z},f_0),
$$
where the first coordinate of the pair is a minimal geodesic of $\p(\h)$ (alas a special one, which leaves the unit vector $f_0\in R(P_0)$ fixed). 
\end{enumerate}
\end{rem}

\section{The restricted sphere bundle}\label{sec5}

In this section we shall consider the restricted version of the canonical bundle by using the restricted Grassmann manifold induced by a (fixed) decomposition $\h=\h_+\oplus\h_-$. Denote by $P_+$ the orthogonal projection onto $\h_+$, and by $P_-$ the projection onto $\h_-$. The restricted Grassmann manifold is the space
$$
\p_2^+(\h)=\{\, P\in\p(\h)\, : \, P-P_+\in\b_2(\h), \, j(P_+,P)=0 \, \},
$$
where $B_2(\h)$ denotes the ideal  of Hilbert-Schmidt operators and $j(P_+,P)$ is the index of the pair of projections $(P_+,P)$ defined by $j(P_+,P)=\text{index}(PP_+:R(P_+) \to R(P))$. It should be said that this space $\p_2^+(\h)$ is in fact the {\it zero index} component of the restricted Grassmannian, or the component containing $P_+$. In \cite{al} the geodesic structure of $\p_2^+(\h)$ was studied.

%({\bf RECORDAR LOS RESULTADOS QUE HAGAN FALTA}).

$\p_2^+(\h)$  is a Hilbert-Riemann manifold, its tangent spaces are naturally embedded in $\b_2(\h)_h$, the (real) Hilbert
space of Hermitian Hilbert-Schmidt operators, endowed with the usual trace. 
The group $\u_2(\h)=\{U\in\u(\h): U-1\in\b_2(\h)\}$ acts transitively on $\p_2^+(\h)$.
The restricted version $\r_+^2$ of the space $\r$ is the set
$$
\r_2^+=\{\, (P,f)\in\r \, : \, P\in\p_+^2(\h) \, \}.
$$
With the same proof as in Proposition \ref{smooth}, restricting the action and the local sections, we obtain that $\r_2^+$ is an embeded, smooth (real analytic) submanifold of the Hilbert space $\b_2(\h)_h\times \h$, and the action is transitive. 

Thus $(T\r_2^+)_{(P_0,f_0)}$ has a natural inner product and Riemannian metric,
\begin{equation}\label{producto interno}
\langle (X,f),(Y,g)\rangle =Tr(XY)+Re<f,g>,
\end{equation}
that we shall call the \textit{ambient} metric.

It is apparent that the tangent space $(T\r_2^+)_{(P_0,f_0)}$  at a given $(P_0,f_0)$ is 
\begin{eqnarray}
(T\r_2^+)_{(P_0,f_0)} &=& \{ (ZP_0-P_0Z, Zf_0): Z\in\b_2(\h), Z^*=-Z\}\nonumber\\
& = & \{(X,f)\in\b_2(\h)_h\times \h: X \hbox{ is } P_0-\hbox{codiagonal and } Re<f,f_0>=0\}\nonumber\\
&\subset & \b_2(\h)_h\times \h.\nonumber
\end{eqnarray}

It will be useful to compute explicitly the orthogonal projection 
$$
\Pi=\Pi_{(P_0,f_0)}:\b_2(\h)_h\times \h\to (T\r_2^+)_{(P_0,Q_0)}\subset \b_2(\h)_h\times \h.
$$
In the proof of Propostion \ref{smooth}, we otained the formula (\ref{E}) of the projection 
$\e$, defined in $\b(\h)\times \h$ onto the tangent space of the whole space $\r$ at $(P_0,f_0)$. Namely
$$
\e(X,f)=(P_0^\perp XP_0+P_0XP_0^\perp, i\text{Im}<h,f_0>f_0 + (P_0 - (f_0 \otimes f_0) )h + P_0^\perp Xf_0).
$$
It is apparent that if $X\in\b_2(\h)_h$, then $\e(X,f)$ belongs to the tangent space of the restricted version $\r_2^+$, and thus $\e$ is a projection onto $(T\r_2^+)_{(P_0,f_0)}$. However $\e$ is not orthogonal. To compute the orthogonal projection $\Pi$ we shall use the following known fact (see \cite[Theorem 1.3]{koli}):
$$
\Pi=\e(\e+\e^*-1)^{-1}.
$$
\begin{prop}
\begin{equation}\label{Pi}
\Pi(X,f)=\left(\Pi_1(X,f),\Pi_2(X,f)\right)
\end{equation}
where
\begin{equation}\label{Pi1}
\Pi_1(X,f)=P_0^\perp XP_0+P_0XP_0^\perp + \frac13 f_0\otimes P_0^\perp f +\frac13 P_0^\perp f \otimes f_0-\frac13 f_0\otimes f_0 X P_0^\perp -\frac13 P_0^\perp X f_0\otimes f_0,
\end{equation}
and
\begin{equation}\label{Pi2}
\Pi_2(X,f)=i Im<f,f_0>f_0+(P_0-f_0\otimes f_0)f+\frac23 P_0^\perp Xf_0+\frac13 P_0^\perp f .
\end{equation} 
\end{prop}
\begin{proof}
Straightforwrd verification.
\end{proof}

\subsection{Levi-Civita connection of the ambient metric}
From now on, for $(P_0,f_0)\in\r_2^+$,  let us denote by $\Pi_{(P_0,f_0)}$ the projection $\Pi$ defined in (\ref{Pi}),(\ref{Pi1}), (\ref{Pi2}).
We endow $\r_2^+$ with the Hilbert-Riemann metric given by the usual trace and real part of the inner product of $\h$:
\begin{defi}\label{metrica 2}
Let $(X,f),(Y,g)\in (T\r_2^+)_{(P_0,f_0)}$. Then we define
$$
\langle (X,f),(Y,g) \rangle_{(P_0,f_0)}=Tr(XY)+Re<f,g>.
$$
\end{defi}

\begin{rem}\label{covarianza Pi}
Let $(P,f)\in\r_2^+$ and $U\in \u_2(\h)$. Then 
$$
\Pi_{U\cdot (P,f)}(X,g)=\Pi_{(P,f)}(U^*XU,U^*g).
$$
This can be verified by means of routine computations, involving elementary facts, such as $UP_0^\perp U^*=(UPU^*)^\perp$, $U(f\otimes f)U^*=Uf \otimes Uf$, etc.
\end{rem} 
Given $(P_0,f_0)\in\r_2^+$, we still denote by $\rho_{(P_0,f_0)}$ the restriction of the former map $\rho_{(P_0,f_0)}$ to the group $\u_2(\h)$, and by $\delta_{(P_0,f_0)}$ its differential at the identity,
$$
\delta_{(P_0,f_0)}:\b_2(\h)_{ah}\to T(\r_2^+)_{(P_0,f_0)}\subset \b_2(\h)_h\times \h , \ \ \delta_{(P_0,f_0)}(Z)=(ZP_0-P_0Z,Zf_0),
$$
where $\b_2(\h)_{ah}$ denotes the space of anti-Hermitian Hilbert-Schmidt operators.
The description of the nullspace of $\delta_{(P_0,f_0)}$ done in the proof of Propostion \ref{smooth}, is valid in this context, assuming that the elements lie in $\b_2(\h)_{ah}$. Namely, $X\in N(\delta_{(P_0,f_0)})$ if $X\in\b_2(\h)$, and it is of the form
$$
X=\left( \begin{array}{cc} X_{11} & 0 \\ 0 & X_{22} \end{array} \right)
$$
with $X_{ii}$ anti-Hermitian and $X_{11}f_0=0$, and  $X_{11}$ can be written as a matrix in terms of the decomposition $R(P_0)=<f_0>\oplus <f_0>^\perp$,
$$
X_{11}=\left( \begin{array}{cccc} 0 & 0 &  0 & \dots \\ 0 & x'_{11} & x'_{12} & \dots \\ 0 & x'_{21} & x'_{22} & \dots \\  \vdots & \vdots & \vdots & \ddots \end{array} \right) =\left( \begin{array}{cc} 0 & 0 \\ 0 & X' \end{array} \right),
$$
where $X'$ is an  anti-Hermitian Hilbert-Schmidt  operator acting in $R(P_0)\ominus <f_0>$. The intersection of $\b_2(\h)$ with the supplement found in Propostion \ref{smooth}, presented as matrices in the decomposition 
$$
\h=<f_0>\oplus (R(P_0)\ominus<f_0>) \oplus N(P_0)
$$
which  are  the elements of the form
\begin{equation}\label{nucleoparaerre reducido}
\left( \begin{array}{cc}
\begin{array}{cc}  i  t & \vec{g} \\  -\vec{\bar{g}}^t  & 0 \end{array}   &  Y \\  -Y^*    & 0 \end{array} \right)=\left(\begin{array}{ccc}  i  t & g & Y_1 \\  -\bar{g}^t  & 0 & Y_2  \\  -Y_1^* & -Y_2^*  & 0 \end{array} \right),
\end{equation}
with $t\in\mathbb{R}$. Denote this set of matrices by $\mathbb{H}_{(P_0,f_0)}$. A straigthforward computation shows that 
\begin{prop}
The space $\hh_{(P_0,Q_0)}$ is the orthogonal complement for $N(\delta_{(P_0,f_0)})$, for the trace inner product in $\b_2(\h)_{ah}$.
\end{prop}

Since the metric is defined by means of a metric in the ambient space, the covariant derivative consists of differentiating in the ambient space and projecting onto the tangent spaces. Therefore a curve $\delta$ is a geodesic if it satisfies the equation
$$
\Pi_{\delta(t)}(\ddot{\delta})=0
$$
for all $t$.
\begin{rem}
Let $Z\in\b_2(\h)_{ah}$. Then $\gamma(t)=e^{tZ}\cdot(P_0,f_0)$ is a geodesic of the connection if and only if
$$
\Pi_{(P_0,f_0)}(Z^2P_0-2ZP_0Z+P_0Z^2,Z^2f_0)=0.
$$
Indeed, $\ddot{\gamma}(t)=(e^{tZ}\{Z_2P_0-2ZP_0Z+P_0Z^2\}e^{-tZ},e^{tZ}Z^2f_0)=e^{tZ}\cdot(Z^2P_0-2ZP_0Z+P_0Z^2,Z^2f_0)$, and thus, using Remark \ref{covarianza Pi},
$$
\Pi_{\gamma(t)}(\ddot{\gamma}(t))=\Pi_{e^{tZ}\cdot(P_0,f_0)}(e^{tZ}\cdot(Z^2P_0-2ZP_0Z+P_0Z^2,Z^2f_0))
$$
$$
=\Pi_{(P_0,f_0)}(Z^2P_0-2ZP_0Z+P_0Z^2,Z^2f_0).
$$
If one takes $Z$ to be $P_0$-codiagonal (and thus in $\hh_{(P_0,f_0)}$), such that $f_0$ is not an eignevector for $Z^2$, then $\gamma$ is not a geodesic.
\end{rem}
\subsection{Reductive connection}

We shall consider another natural linear connection in $\r_2^+$, induced by the decomposition of the Banach-Lie algebra $\b_2(\h)$ of $\u_2(\h)$, induced by  $(P_0,f_0)\in\r_2^+$: $\b_2(\h)_{ah}=N(\delta_{(P_0,f_0)})\oplus \hh_{(P_0,f_0)}$.

Let us denote by $\ii_2(P_0,f_0)$ the isotropy group of the restricted action 
$$
\ii_2(P_0,f_0)=\{V\in\u_2(\h): V\cdot(P_0,f_0)=(P_0,f_0)\}.
$$ 
Note that if $V\in\ii_2(P_0,f_0)$, then $Ad(V)(N(\delta_{(P_0,f_0)})=N(\delta_{(P_0,f_0)})$. Indeed, let $Y\in N(\delta_{(P_0,f_0)})$, 
then $YP_0=P_0Y$ and $Yf_0=0$. Since $V\in\ii_2(P_0,f_0)$, $VP_0=P_0V$ and $Vf_0=f_0$ (thus, $V^*f_0=f_0$). Then
$VYV^*$ also commutes with $P_0$ and $VYV^*f_0=VYf_0=0$. On the other hand, $Ad(V)$ is an orthogonal transformation of $\b_2(\h)$ (endowed with the (real part of the) trace inner product). Therefore $Ad(V)$ leaves $\hh_{(P_0,f_0)}$ ($=N(\delta_{(P_0,f_0)})^\perp$) invariant. 

Let us prove  that the distribution 
$$
\r_2^+ \ni (P_0,f_0) \mapsto \hh_{(P_0,f_0)}\subset \b_2(\h)_{ah}
$$
is {\it smooth}. In other words, if we denote by ${\bf P}_{(P_0,f_0)}$ the $Re Tr$-orthogonal projection onto $\hh_{(P_0,f_0)}$, we must prove  that the mapping
$$
\r_2^+ \ni (P_0,f_0) \mapsto {\bf P}_{(P_0,f_0)} \in \b_{\mathbb{R}}(\b_2(\h)_{ah})
$$
is smooth, where $\b_{\mathbb{R}}(\b_2(\h)_{ah})$ denotes the Banach space of real linear bounded operators acting in $\b_2(\h)_{ah}$ (endowed with the Hilbert-Schmidt norm).
\begin{lem}
Let ${\bf Q}={\bf P}^\perp_{(P_0,f_0)}=1_{\b_2(\h)_{ah}}-{\bf P}_{(P_0,f_0)}$. Then
$$
{\bf Q}(X)=P_0XP_0+P_0^\perp XP_0^\perp+f_0\otimes f_0 X f_0\otimes f_0-f_0\otimes f_0 X P_0- P_0 X f_0\otimes f_0.
$$
In particular, the mapping $\r_2^+\to \b_{\mathbb{R}}(\b_2(\h)_{ah})$, $(P_0,f_0)\mapsto {\bf P}_{(P_0,f_0)}$ is smooth.
\end{lem}
\begin{proof}
Since $f_0\in R(P_0)$, clearly ${\bf Q}(X)$ commutes with $P_0$. Also it is clear that if  $X$ anti-Hermitian and Hilbert-Schmidt, then so is ${\bf Q}(X)$. Finally, note that ${\bf Q}(X)f_0=0$. With this formula, it is clear that the map $(P_0,f_0)\mapsto {\bf P}_{(P_0,f_0)}$ is smooth.
\end{proof}
These facts imply:
\begin{coro}
The distribution of supplements $\r_2^+\ni (P,f)\mapsto \hh_{(P,f)}$ defines a reductive structure in $\r_2^+$ (as a $\u_2(\h)$-homogeneous space). 
\end{coro}

\begin{rem}
For $(P_0,f_0)\in\r_2^+$, we will denote
$$
\delta_0=\delta_{(P_0,f_0)}=(\rho_{(P_0,f_0)})_{*1},
$$
then the map $\delta_0|_{\hh_{(P_0,f_0)}}:\hh_{(P_0,f_0)}\to (T\r_2^+)_{(P_0,f_0)}$
is a linear isomorphism. Let us denote by $\kappa_{(P_0,f_0)}$ its inverse. The distribution $(P,f)\mapsto \kappa_{(P,f)}$ is usually called the 1-form of the reductive structre, and is used to define the invariants of the linear connection. 

\end{rem}
Accordingly, we define a Hilbert-Riemann metric which is coherent with this structure: we postulate that the 1-form to be isometric.
That is, 
\begin{defi}\label{metrica reductiva}
If $(X,g)\in(T\r_2^+)_{(P_0,f_0)}$, then 
$$
|(X,g)|^2_{r,(P_0,f_0)}=-Tr\left((\kappa_{(P_0,f_0)}(X,g))^2\right),
$$
i.e., $|(X,g)|_{r,(P_0,f_0)}=\|Z\|_2$, where $Z\in\hh_{(P_0,f_0)}$ is the unique element such that $\delta_{(P_0,f_0)}(Z)=(X,g)$.
\end{defi}

Note that we endowed $\r_2^+$ with the \textit{quotient metric}, giving it the Riemannian metric such that the map $\rho_0=\rho_{(P_0,f_0)}:\u_2(\h)\to \r_2^+$ is a \textit{Riemannian submersion}. That is, $\rho_0$  is a smooth submersion such that for any $U\in \u_2(\h)$, if $(P,f)=\rho(U)$, then $\delta=\delta_{(P,f)}$ is an isometry between $\hh_{(P,f)}$ and $(T\r_2^+)_{(P,f)}$.

\begin{rem}
The metric defined above in Defintion \ref{metrica reductiva} coincides with the quotient metric, as defined in (\ref{norma de R}), if one replaces the operator norm by the Hilbert-Schmidt norm.
\end{rem}

If $V\in (T\r_2^+)_{(P,f)}$, then $\kappa_{(P,f)}(V)\in \hh_{(P,f)}$ is its unique \textit{horizontal lift};  and if $X$ is a smooth vector field in $\r_2^+$ then $X^h=\kappa\circ X: u\mapsto \kappa_{\rho(u)}(X_{\rho(u)})$ is a smooth vector field in $\u_2(\h)$, the \textit{horizontal lift} of $X$. For obvious reasons $N(\delta)$, the kernel of $\delta=\rho_*$ is also known as the \textit{vertical space}, and vectors in that kernel are named  \textit{vertical vectors}.

\begin{rem}It is apparent from the definition of quotient metric that if $V=\delta_0(Z)=(ZP_0-P_0Z,Zf_0)\in (T\r_2^+)_{(P_0,f_0)}$ with $Z$ a horizontal vector described as in (\ref{nucleoparaerre reducido}), then
\begin{eqnarray}
\|X\|_q^2 &= & \|Z\|^2=Tr \left(\begin{array}{ccc}  t^2 +g\cdot\overline{g}^t+Y_1Y_1^* & \dots & \dots \\  \dots & \overline{g}^t\cdot g+ Y_2Y_2^* & \dots  \\  \dots & \dots & Y_1Y_1^*+Y_2Y_2^* \end{array} \right)\nonumber\\
&=& t^2+ 2\|g\|^2+2\|Y_1\|_2^2+2\|Y_2\|_2^2.\nonumber
\end{eqnarray}
On the other hand, if we endow $\r_2^+$ with the \textit{ambient} metric of Definition \ref{metrica 2}, we obtain
\begin{eqnarray}
\|X\|_a^2 &=&\|\delta(Z)\|_2^2=\|ZP_0-P_0Z\|^2+\|Zf_0\|^2\nonumber\\
&=& \left\|\left(\begin{array}{ccc}  0 & 0 & Y_1 \\  0 & 0 & Y_2  \\  Y_1^* & Y_2^* & 0 \end{array}\right)\right\|_2^2+\|(it \; g\; Y_1)\|^2\nonumber\\
&=& Tr \left(\begin{array}{ccc}  Y_1Y_1^* & \dots & \dots \\  \dots &  Y_2Y_2^* & \dots  \\  \dots & \dots & Y_1^*Y_1+Y_2^*Y_2 \end{array} \right)+t^2+\|g\|^2+\|Y_1\|^2\nonumber\\
&=& t^2+\|g\|^2+3\|Y_1\|_2^2+2\|Y_2\|_2^2.\nonumber
\end{eqnarray}
Therefore
$$
\frac{1}{\sqrt{2}}\|X\|_q\le \|X\|_a \le \sqrt{\frac{3}{2}} \|X\|_q
$$
which makes both Riemannian metrics equivalent (the inequalities are sharp). Consequently, if $d$ denotes the Riemannian distance given by taking the infima of the lengths of paths joining given endpoints, both distances are equivalent with constants
\begin{equation}\label{unieq}
\frac{1}{\sqrt{2}}\,d_q(x,y)\le  d_a(x,y) \le  \sqrt{\frac{3}{2}} \, d_q(x,y)
\end{equation}
for all $x,y\in \r_2^+$.
\end{rem}

\begin{coro}The distance induced by the Levi-Civita connection of the quotient metric induces in $\r_2^+$ the same ambient topology of its embedding in the Hilbert space $\b_2(\h)_h\times \h$.
\end{coro}
\begin{proof}
Combine the equivalence of distances of the previous remark, with the fact that $\r_2^+$ is an embedded submanifold, and use that the Riemannian distance gives the manifold topology \cite[Proposition 6.1]{lang}.
\end{proof}

\medskip

We collect below some known results we will use on Riemannian submersions:

\begin{teo}Let $\rho: (U,g_0)\to (R,g)$ be a Riemannian submersion, let $X,Y$ be smooth vector fields in $(R,g)$. Then
\begin{enumerate}
\item If $Z_0$ is a smooth vertical field in $(U,g_0)$, then 
$$
[X,Y]^v:=[X^h,Y^h]-[X,Y]^h\textrm{ and } [X^h,Z_0]
$$
are smooth vertical vector fields in $(U,g_0)$ (here $[\cdot,\cdot]$ denotes the usual Lie bracket of smooth vector fields).
\item If $\nabla^0$ is the Levi-Civita connection of $(U,g_0)$ and $\nabla$ is the Levi-Civita connection of $(R,g)$, then
$$
(\nabla_XY)_{\rho(u)}=\delta(\nabla^0_{X^h}Y^h)
$$
where $\delta=\rho_*$ and moreover
$$
\nabla^0_{X^h}Y^h=(\nabla_XY)^h+\frac{1}{2}[X,Y]^v.
$$
\item Let $\Gamma$ be a geodesic of $(U,g_0)$, with $\Gamma'(0)$ horizontal, then $\Gamma'(t)$ is horizontal for all $t$ in the domain of $\Gamma$, and $\gamma=\rho(\Gamma)$ is a geodesic of $(R,g)$ of the same length of $\Gamma$. Conversely, if $\gamma$ is a geodesic of $(R,g)$ with $\gamma(0)=\rho(u)$, there exists a unique horizontal lift $\Gamma$ of $\gamma$ such that $\Gamma(0)=u$ and $\Gamma$ is a geodesic of $(U,g_0)$.
\item The exponential maps are thus related by $Exp^R(v)=\rho\circ\exp^U(v^h)$.
\item If $(U,g_0)$ is complete (resp. geodesically complete) then $(R,g)$ is complete (resp. geodesically complete).
\item If $X,Y$ are orthonormal, the sectional curvatures are related by O'Neill's formula.
$$
K^R(X,Y)=K^U(X^h,Y^h)+\frac{3}{4}\|[X,Y]^v\|^2.
$$
\item In particular, if $R=U/I$ is a quotient of Lie groups and the metric $g_0$ on $U$ is bi-invariant, then geodesics of $(R,g)$ are given by $\rho(e^{tZ})$ where $Z$ is an horizontal vector, and sectional curvature is positive and given by
$$
K^R(v,w)=\frac{1}{4}\|[v,w]^h\|^2+\|[v,w]^v\|^2,
$$
if $v,w\in (TR)_{\rho(1)}$ and now $[\cdot,\cdot]$ denotes the usual Lie algebra bracket.
\end{enumerate}
\end{teo}
%\cite{gallot}, Lemma 3.54 (vert), Proposition 3.55 (conection), Proposition 2.109 (geos), Theorem 3.61 (curv), Theorem 3.65 (geo groups)
\begin{proof}
See for instance \cite{gallot} Lemma 3.54, Proposition 3.55, Proposition 2.109, Theorem 3.61, Theorem 3.65.
\end{proof}

\begin{rem} If $Z(t)$ is a tangent field along a curve $\gamma(t)\in\r_2^+$, the covariant derivative $D_{\dot{\gamma}}Z$ can be computed using the formula in the second item above, and the formula for the covariant derivative in the unitary group
$$
(\nabla^0_X Y)_u=DY_u(X_u)+\frac{1}{2}\left[X_uY_u^* u+uY_u^*X_u  \right].
$$
Combining these formulas, we obtain
$$
D_{\dot{\gamma}}Z=\delta_{\gamma}\left(\frac{d}{dt}\kappa_{\gamma}(Z)-\frac{1}{2}\left[\kappa_{\gamma}(\dot{\gamma})\kappa_{\gamma}(Z)+\kappa_{\gamma}(Z)\kappa_{\gamma}(\dot{\gamma}) \right] \right).
$$
\end{rem}

\begin{rem}
Since the geodesics of the unitary group are the one-parameter groups $t\mapsto e^{tZ}$, it follows from the fourth item above that geodesics or $\r_2^+$ are exactly the image through $\rho$ of the one-parameter groups with horizontal speed, i.e. $D_{\dot{\gamma}}\dot{\gamma}=0$ and $\gamma(0)=(P,f)$ if and only if $\gamma(t)=\rho(e^{tZ})$ with $Z\in \hh_{(P,f)}$

The exponential map of $\r_2^+$ with the quotient metric is then given by 
$$
Exp(V)=Exp_{(P_0,f_0)}(V)=\rho(e^{\kappa_{(P_0,f_0)}V}),
$$
or equivalently, if $V=(ZP_0-P_0Z,Zf_0)$ for horizontal $Z$, then
\begin{equation}\label{riemexp}
Exp : V=(ZP_0-P_0Z,Zf_0)\mapsto (e^ZP_0e^{-Z},e^Zf_0).
\end{equation}

Since $\delta_0=\delta|_{\hh_{(P_0,f_0)}}:\hh_{(P_0,f_0)}\to (\r_2^+)_{(P_0,f_0)}$ is an isometric isomorphism, we can equivalently study the map  $E=Exp\circ\kappa=\rho\circ \exp:\hh_{(P_0,f_0)}\to \r_2^+$, that is
\begin{equation}\label{riemexpsimpli}
E: Z\mapsto (e^ZP_0e^{-Z},e^Zf_0)=\rho(e^Z).
\end{equation}
\end{rem}

If $exp_{*Z}(W)$ denotes the differential at $Z$ of the Lie-group exponential, $e^Z=\sum_{n\ge 1}\frac{1}{n!}Z^n$, evaluated in $W$,  then it is well-known that
$$
exp_{*Z}(W)=e^Z(F(ad Z)W)=\int_0^1 e^{(1-s)Z}We^{sZ}ds=(G(ad Z)W)e^{Z}.
$$
Here $F,G:\mathbb C\to\mathbb C$ denote the entire functions $F(\lambda)=\frac{1-e^{-\lambda}}{\lambda}, F(\lambda)=\frac{e^{\lambda}-1}{\lambda}$, note that $G(-\lambda)=-F(\lambda)$. Therefore
$$
E_{*Z}(W)=( e^Z(F(ad Z)W)P_0e^{-Z}+e^ZP_0(G(-ad Z)W)e^{-Z}, e^Z(F(ad Z)W)f_0),
$$
and thus
\begin{equation}\label{riemexpdif}
e^{-Z}\cdot E_{*Z}(W)=( (F(ad Z)W)P_0-P_0(F(ad Z)W),(F(ad Z)W)f_0)=\delta_0 (T(W))
\end{equation}
where we denote $T=T_Z=F(ad Z): \b_2(\h)_{ah}\to \b_2(\h)_{ah} $ for short.  Note that $T$ it is a bounded linear operator.

\begin{rem}
Since the roots of $F$ are located at $\{\lambda=2k\pi i, k\in \mathbb Z\ne 0\}$, and $\|ad Z\|\le 2\|Z\|$, it is well-know that as long as $\|Z\|<\pi$, $T$ is an invertible operator. In what follows we need to refine this remark to obtain better control on the range of $T$ acting on $\hh$. Consider the real even function
$$
g(t)=\frac{\sin(t)}{t}+\frac{\cos(t)-1}{t^2},
$$
extended as $g(0)=1/2$ to make it continuous. Then the first root of $g$ will be denoted by $r_0>0$, a straightforward computation indicates that $\frac{\pi}{2}<r_0<\frac{2\pi}{3}$ and $g$ is strictly decreasing in $(0,r_0)$ (and increasing in $(-r_0,0)$ by symmetry).
\end{rem}

\begin{lem}
For each $Z\in \hh_{(P_0,f_0)}$, let $T=F(ad Z)$. Then $T$ is a normal contraction acting on $\b_2(\h)_{ah}$, and moreover if $\|Z\|<\frac{r_0}{2}$,  then $\|T-1\|<1$.
\end{lem}
\begin{proof}
From the definition of $ad Z(W)=ZW-WZ$ and the trace inner product in the Hilbert space $\b_2(\h)_{ah}$, it follows that $ad Z$ is anti-Hermitian for any anti-Hermitian $Z$. Therefore $T=F(ad Z)$ is a normal operator; it is a contraction because
$$
\|T\|=\|F(ad Z)\|=\|\int_0^1 e^{s \,ad Z}ds\|\le \int_0^1 \|e^{s \,ad Z}\|ds\le 1.
$$
Here, since $s \, ad Z$ is anti-Hermitian, $e^{s\, ad Z}$ is unitary. Now since $T$ is normal, its norm can be computed using its spectral radius, and for $it\in\sigma(Z)$ we obtain
\begin{eqnarray}
|F(it)-1|^2 &=& |\frac{1-e^{-it}}{it}-1|^2=\frac{1}{t^2}|1-e^{-it}-it|^2=\frac{1}{t^2}|1-\cos t+i(\sin t-t)|^2\nonumber\\
&=&\frac{1}{t^2} (1+\cos^2t-2\cos t+\sin^2t +t^2-2t\sin t)\nonumber\\
&=& 1+2(\frac{1-\cos t}{t^2}-\frac{\sin t}{t})=1-2g(t).
\end{eqnarray}
This tells us that $\|F(ad Z)-1\|<1$ when $|t|<r_0$. If $\|Z\|<r_0/2$ then $\|ad Z\|<r_0$,  then $\sigma(ad Z)\subset (-ir_0,ir_0)$ and the conclusion of the lemma follows.
\end{proof}

\begin{coro}\label{difiso}
Let $(P,f)=U\cdot (P_0,f_0)\in \r_2^+$, let $Exp_(P,f):(T\r_2^+)_{(P,f)}\to \r_2^+$ be the Riemannian exponential of the quotient metric, given by 
$$
Exp_{(P,f)}(V)=\rho_{(P,f)}(e^{\kappa_{(P,f)}V}).
$$
Then its differential is an isomorphism in the ball of radius $r_0/2$ around $0\in (T\r_2^+)_{(P,f)}$.
\end{coro}
\begin{proof}
Since the action of the unitary group is an isometric isomorphism of the Riemannian manifold $\r_2^+$, it will suffice to prove the assertion for $(P_0,f_0)$. Since $\delta_0$ is an isometric isomorphism, it will be more convenient to prove this assertion of the simplified exponential map $E:\hh_{(P_0,f_0)}\to \r_2^+\subset \b_2(\h)\times \h$ given by equation (\ref{riemexpsimpli}).  We have to show that the differential of $E$ is invertible  in a certain ball; since the action of the unitary group is isometric, it suffices to prove that the map given in (\ref{riemexpdif}) is an isomorphism for $V=(X,g)=\delta_0(Z)=(ZP_0-P_0Z,Zf_0)$ if $\|V\|_q=\|Z\|<r_0/2$ (here $Z$ is the unique horizontal lift of $V$). 

In what follows we denote $\hh=\hh_{(P_0,f_0)}$ and ${\bf P}={\bf P}_{(P_0,f_0)}$ (the orthogonal projection onto the horizontal space $\hh$), Let $A=T{\bf P}+(\bf{1-P})$ which is a bounded operator acting on the Hilbert space $\b_2(\h)_{ah}$ ($T=F(ad X)$ as before). Then $\|A-1\|=\|(T-1){\bf P}\|\le \|1-T\|=c_0<1$ therefore $A$ is invertible, which shows that $\b_2(\h)_{ah}=T\hh+ \ker(\delta)=T\hh+ \hh^{\perp}$. On the other hand if $T{\bf P}\xi=({\bf 1-P})\eta$, 
$$
\|{\bf P}\xi\|^2>\|(1-T){\bf P}\xi\|^2=\|T{\bf P}\xi-{\bf P}\xi\|^2=\|({\bf 1- P})\eta-{\bf P}\xi\|^2=\|({\bf 1-P})\eta\|^2+\|{\bf P}\xi\|^2,
$$
i.e.  $({\bf 1-P})\eta=0$, which proves that $T\hh\cap \hh^{\perp}=\{0\}$.

Therefore $\b_2(\h)_{ah}$ is the direct sum of the kernel of $\delta_0$ (the vertical space, $\hh^{\perp}$) and  $T\hh$,
$$
\b_2(\h)_{ah}=T\hh \,\dot{+}\, \hh^{\perp}.
$$
This implies that $\delta_0|_{T\hh}$ is still an isomorphism onto $(T\r_2^+)_{(P_0,f_0)}$. By equation (\ref{riemexpdif}), this proves the claim that $E_{*Z}$ is an isomorphism.
\end{proof}

Recall the Riemannian distances $d_q,d_a$ in $\r_2^+$ are defined as the infima of the lengths of the piecewise smooth paths joining given endpoints, when the curves are measured with the quotient metric (resp. the ambient metric). Regarding the minimality of geodesics, the injectivity radius, and the metric completeness, the following can be established. 

\begin{teo}
Let $(P_0,f_0)\in\r_2^+$, endowed with the quotient metric given in Definition \ref{metrica reductiva}.
\begin{enumerate}
\item
 For any $(X,g)\in(T\r_2^+)_{(P_0,f_0)}$, let $Z\in\hh_{(P_0,f_0)}$ the unique element such that $\delta_{(P_0,f_0)}(Z)=(X,g)$. Then 
$$
\delta(t)=e^{tZ}\cdot(P_0,f_0),
$$
which is the unique geodesic of the reductive connection which satisfies $\delta(0)=(P_0,f_0)$ and $\dot{\delta}(0)=(X,g)$, has minimal length along its path for $t$ such that $|t|\le \frac{\pi}{4\|Z\|_2}$. This geodesic is unique if $|t|<\frac{\pi}{4\|Z\|_2}$.
\item The space $\r_2^+$ is complete with both its metrics $d_q,d_a$ (quotient and ambient).
\item
Let $(P_1,f_1)\in\r_2^+$ such that $d_q((P_1,f_1),(P_0,f_0))<\pi/4$. Then there exists a unique geodesic $\delta$ of the reductive connection such that $\delta(0)=(P_0,f_0)$ and $\delta(1)=(P_1,f_1)$. This geodesic has minimal length (among smooth curves joining $(P_0,f_0)$ and $(P_1,f_1)$).
\end{enumerate}
\end{teo}
\begin{proof}
The first assertion is a consequence of the estimation for the injectivity radius of homogeneous spaces of the $p$-Schatten unitary groups given in \cite[Theorem 4.9]{agc10}. From that theorem also follows the uniqueness when $4\|Z\|_2 \, |t|<\pi$.

Regarding the second assertion, it was also proven in  \cite[Theorem 5.5]{agc10} that homogeneous spaces of the restricted unitary group are complete with the quotient metric. Therefore $(\r_2^+,d_q)$ is complete. Since both metrics are uniformly equivalent by (\ref{unieq}), $(\r_2^+,d_a)$ is also complete.

Now note that the first assertion says that the Riemannian exponential is injective in the open ball of (quotient) radius $\pi/4$. Combining this with Corollary \ref{difiso}, it follows that the injectivity radius of the Riemannian exponential of $\r_2^+$ is at least $\pi/4$ (since $r_0/2>\pi/4$). Therefore by the theory of Hilbert-Riemann manifolds (see for instance \cite[Theorem 6.4]{lang}), the third assertion follows.
\end{proof}

\medskip

\section*{Acknowledgements} This research was supported by CONICET (PIP 2014 11220130100525) and ANPCyT (PICT 2015 1505).

{\sc (Esteban Andruchow)} {Instituto de Ciencias,  Universidad Nacional de Gral. Sar\-miento,
J.M. Gutierrez 1150,  (1613) Los Polvorines, Argentina and Instituto Argentino de Matem\'atica, `Alberto P. Calder\'on', CONICET, Saavedra 15 3er. piso,
(1083) Buenos Aires, Argentina.}

\noindent e-mail: {\sf eandruch@ungs.edu.ar}

\bigskip

{\sc (Eduardo Chiumiento)} {Departamento de de Matem\'atica, FCE-UNLP, Calles 50 y 115, 
(1900) La Plata, Argentina  and Instituto Argentino de Matem\'atica, `Alberto P. Calder\'on', CONICET, Saavedra 15 3er. piso,
(1083) Buenos Aires, Argentina.}     
                                               
\noindent e-mail: {\sf eduardo@mate.unlp.edu.ar}

\bigskip

{\sc (Gabriel Larotonda)} {Instituto de Matem\'atica, `Alberto P. Calder\'on' (CONICET),
\\ and Departamento de Matem\'atica, Facultad de Cs. Exactas y Naturales, Universidad de Buenos Aires. Ciudad Universitaria (1428) CABA, Argentina }

\noindent e-mail: {\sf glaroton@dm.uba.ar }


\begin{thebibliography}{XX}


\bibitem{al} E. Andruchow, G. Larotonda,  {\em Hopf-Rinow theorem in the Sato Grassmannian}, J. Funct. Anal. 255 (2008), no. 7, 1692--1712.

\bibitem{AL10} E. Andruchow, G. Larotonda, {The rectifiable distance in the unitary Fredholm group}, Studia Math. 196 (2010), 151--178.

\bibitem{agc10} E. Andruchow, G. Larotonda, L. Recht, {\em Finsler geometry and actions of the $p$-Schatten unitary groups}, Trans. Amer. Math. Soc. 362 (2010), 319--344.

\bibitem{ARV} E. Andruchow, L. Recht, A. Varela, {\em Metric geodesics of isometries in a Hilbert space and the extension problem}, Proc. Amer. Math. Soc.  135  (2007), 2527-2537.


\bibitem{beltita} D. Belti\c{t}$\breve{\text{a}}$, {\it Smooth homogeneous structures in operator theory}, Chapman and Hall/CRC, Monographs and Surveys in Pure and Applied Mathematics 137, 2006.

\bibitem{BR07b} D. Belti\c{t}$\breve{\text{a}}$, T. Ratiu, A. Tumpach, {\it The restricted Grassmannian, Banach Lie-Poisson spaces and coadjoint orbits}, J.  Funct. Anal. 247 (2007), no. 1, 138-168.


\bibitem{BV17} T. Bottazzi, A. Varela, {\em Unitary subgroups and orbits of compact self-adjoint operators}, Studia Math. 238 (2017), 155--176.

\bibitem{Ch10} E. Chiumiento, {\em Geometry of $\mathfrak{I}$-Stiefel manifolds}, Proc. Amer. Math. Soc. 138 (2010), no. 1,  341-353.

\bibitem{CPR93} G. Corach, H. Porta, L. Recht,  {The geometry of spaces of projections in $C^*$-algebras},  Adv. Math. 101 (1993), no. 1, 59--77.

\bibitem{DKW82} C.  Davis,  W.M.  Kahan,  and  H.F.  Weinberger,
{\em Norm-preserving  dilations  and  their  applications to optimal error bounds}, SIAM J. Numer. Anal. 19 (1982), pp. 445--469.

\bibitem{cocometricacociente}  C. E. Dur\'an, L.E. Mata-Lorenzo, L. Recht, {\em Metric geometry in homogeneous spaces of the unitary group of a $C^*$-algebra. I. Minimal curves}, Adv. Math. 184 (2004), no. 2, 342--366.

\bibitem{gallot} S. Gallot, D. Hulin, J. Lafontaine, {\em Riemannian geometry}. Third edition. Universitext. Springer-Verlag, Berlin, 2004.

\bibitem{kn} S. Kobayashi, K. Nomizu. {\em Foundations of differential geometry. Vol. I.} Reprint of the 1963 original. Wiley Classics Library. A Wiley-Interscience Publication. John Wiley \& Sons, Inc., New York, 1996.

\bibitem{koli} J. J. Koliha, {\em Range projections of idempotents in $C^*$-algebras}. Demonstratio Math. 34 (2001), no. 1, 91--103.

\bibitem{lang} S. Lang, {\em Differential and Riemannian manifolds}. Third edition. Graduate Texts in Mathematics, 160. Springer-Verlag, New York, 1995.

\bibitem{K79} Z.V. Kovarik, {\em Manifolds of linear involutions}, Linear Algebra Appl. 24
(1979), 271--287.

\bibitem{krein} M.G. Krein, {\em The theory of self-adjoint extensions of semibounded Hermitian transformations
and its applications}, Mat. Sb. 20 (1947), 431--495; 21 (1947), 365--404 (in Russian).

\bibitem{m-l r} L.E. Mata-Lorenzo, L.  Recht, {\em Infinite-dimensional homogeneous reductive spaces}, Acta Cient. Venezolana 43 (1992), no. 2, 76--90. 

\bibitem{milnor} J. W. Milnor, J. D. Stasheff, {\em Characteristic classes}, Princeton University Press, Princeton, N. J., 1974.


\bibitem{ps} A. Pressley, G. Segal, {\em Loop groups}, Oxford Mathematical Monographs. Oxford Science Publications. The Clarendon Press, Oxford University Press, New York, 1986.


\bibitem{pr} H. Porta, L. Recht, {\em Minimality of geodesics in Grassmann manifolds}, Proc. Amer. Math. Soc. 100 (1987), 464--466.


\bibitem{rae} I. Raeburn, {\em The relationship between a commutative Banach algebra and its maximal ideal space}, J. Funct. Anal. 25 (1977), no. 4, 366--390.

\bibitem{RN55} F. Riesz, B. Sz.-Nagy, {\em Functional Analysis}, Ungar, New York, 1955.


\bibitem{S44}  N.E. Steenrod, {\em The classification of sphere bundles}, Annals of Math. 45 (1944), no. 2, 294--311.

\bibitem{W40} H. Whitney, {\em On the theory of sphere-bundles}, Proc. Natl. Acad. Sci. USA  26 (1940), no. 2,  148--153.

\end{thebibliography}
\end{document}